\RequirePackage[english]{babel}
\documentclass
{amsart}

\usepackage[OT1]{fontenc}    %
\usepackage{type1cm}         

\usepackage{comment}
\usepackage{enumerate}

\usepackage{amsthm}
\RequirePackage{amsmath,amsfonts,amssymb,amsthm}

\RequirePackage[dvips]{graphicx}

\usepackage{psfrag}    
\usepackage[usenames]{color}
  \numberwithin{equation}{section}

  \newcommand{\N}{\mathbb{N}}         
  \newcommand{\R}{\mathbb{R}}         
  \newcommand{\Rn}{\R^n}              
  \newcommand{\Le}{\mathcal{L}}       

  \DeclareMathOperator{\dist}{dist}

  \newtheorem{thm}{Theorem}[section]
  \newtheorem{lemma}[thm]{Lemma}
  \newtheorem{prop}[thm]{Proposition}
  \newtheorem{cor}[thm]{Corollary}

  \newtheorem{claim}{Claim}

  \theoremstyle{remark}
  \newtheorem{rem}[thm]{Remark}
  \newtheorem{rems}[thm]{Remarks}
  \newtheorem{ex}[thm]{Example}

\begin{document}

\title{On Cantor sets and doubling measures}

\author[M.\ Cs\"ornyei]{Marianna Cs\"ornyei}
\address{Department of Mathematics\\
5734 S. University Avenue\\
Chicago, Illinois 60637\\
USA}
\email{csornyei@math.uchicago.edu}

\author[V.\ Suomala]{Ville Suomala}
\address{Department of Mathematical Sciences\\
P.O. Box 3000\\
FI-90014 University of Oulu\\
 Finland} 
\email{ville.suomala@oulu.fi}

\thanks{The first author was supported by OTKA grant no. 72655. The second author was supported by the Academy of Finland
  (project \#126976).} 

\begin{abstract}
For a large class of Cantor sets on the real-line, we find sufficient
and necessary conditions implying that  
a set has positive (resp. null) measure for all doubling measures of the
real-line. 
We also discuss same type of questions for atomic doubling
measures defined on certain midpoint Cantor sets.
\end{abstract}

\maketitle

\section{Introduction and notation}

Our main goal in this paper is to study the size of Cantor sets on
the real-line $\R$ from the point of view of doubling measures. Recall that
a measure $\mu$ on a metric space $X$ is called
\emph{doubling} if 
there is a constant $c<\infty$ such that 
\[0<\mu(B(x,2r))\le c\mu(B(x,r))<\infty\]
for all $x\in X$ and $r>0$. Here $B(x,r)$ is the open ball with centre
$x\in X$ and radius 
$r>0$. 
We note that the collection of doubling measures on $\R$, and more
generally, on any complete doubling metric space where isolated points
are not dense, is rather rich. For instance, given $\varepsilon>0$,
there are doubling measures on $\R$ having full measure on a
set of Hausdorff and packing dimension at most $\varepsilon$. See
\cite{he}, \cite{t}, \cite{wu}, \cite{krs}.

Let
$\mathcal{D}(\R)$ be the collection of all doubling measures on $\R$ and denote
\begin{align*}
\mathcal{T}=\{C\subset\R\,:\,\mu(C)&=0\text{ for all
}\mu\in\mathcal{D}(\R)\},\\  
\mathcal{F}=\{C\subset\R\,:\,\mu(C)&>0\text{ for all }\mu\in\mathcal{D}(\R)\}.
\end{align*}
In the literature, the sets in $\mathcal{F}$ have
been called quasisymmetrically thick \cite{sw}, \cite{he}, thick for doubling
measures \cite{hww}, and very fat \cite{bhm} and those in
$\mathcal{T}$ have been termed quasisymmetrically null \cite{sw}, \cite{he}, null
for doubling measures \cite{hww}, and thin \cite{bhm}. 
We call $C\subset\R$ \emph{thin} if
$C\in\mathcal{T}$ and \emph{fat} if $C\in\mathcal{F}$.  

In this paper, we address the problems of finding
 sufficient and/or necessary
conditions for a Cantor set $C\subset\R$ to be fat (resp. thin). These
problems arise naturally from the study of compression and expansion properties
of quasisymmetric maps $f\colon\R\rightarrow\R$, see
\cite[13.20]{he}. 
A related problem is to characterise
those subsets $U\subset\R$ which carry
nontrivial doubling measures (\cite[Open problem 1.18]{he03}); If
$C\subset\R$ is a fat Cantor set, then it is easy to see that
$U=\R\setminus C$ does not carry nontrivial doubling measures. For if it did,
then one could extend any doubling measure $\mu$ on $U$ to $\R$ by letting $\mu(C)=0$, and this would
contradict $C$ being fat. 

We begin by discussing thinness and fatness for the middle interval
Cantor sets $C(\alpha_n)$ determined via sequences
$(\alpha_n)_{n=1}^\infty$, $0<\alpha_n<1$, as follows: 
We first remove an
open interval of length $\alpha_1$ from the
middle of $I_{1,1}=[0,1]$ and denote the remaining two intervals by $I_{2,1}$ and $I_{2,2}$. At the $k$:th step, $k\geq 2$, we have
$2^{k-1}$ intervals $I_{k,1},\ldots, I_{k,2^{k-1}}$ of length
$\ell_k=2^{-k+1}\prod_{n=1}^{k-1}(1-\alpha_n)$ and we remove an
interval of length $\alpha_{k}\ell_k$ from the middle of each
$I_{k,i}$. 
Finally, the \textit{middle interval Cantor set $C=C(\alpha_n)$} is
defined by 
\[C=\bigcap_{k\in\N}\bigcup_{i=1}^{2^k}I_{k,i}.\] 

The theorem below follows by combining results of Wu \cite[Theorem 1]{wu2},
Staples and Ward \cite[Theorem 1.4]{sw}, 
and Buckley, Hanson, and MacManus \cite[Theorem 0.3]{bhm}. For
$0<p<\infty$, we denote by $\ell^p$ the set of all sequences $(\alpha_n)_{n=1}^\infty$,
$0<\alpha_n<1$, 
for which $\sum_{n=1}^\infty\alpha_{n}^p<\infty$.
\begin{thm}\label{thm:intro}
Let $C=C(\alpha_n)$. Then
\begin{enumerate}
\item\label{symmthm1} 
$C$ is thin if and only if $(\alpha_n)\notin\bigcup_{0<p<\infty}\ell^p$.
\item\label{symmthm2} 
$C$ is fat if and only if
  $(\alpha_n)\in\bigcap_{0<p<\infty}\ell^p$. 
\end{enumerate}
\end{thm}

In a recent paper, Han, Wang, and Wen \cite{hww} generalised Theorem
\ref{thm:intro} for a broader collection of (still very symmetrtic)
Cantor sets. 
Related results on thin and fat sets may be found in \cite{wu2},
\cite{kw}, \cite{wu}, \cite{sw}, \cite{bhm}, 
\cite{www}, and \cite{ors}.

The known proofs for Theorem \ref{thm:intro} and its generalisation in
\cite{hww} rely heavily on the symmetries of the sets
$C(\alpha_n)$. In this paper, we wish to consider analogues of Theorem \ref{thm:intro} for
Cantor sets with much less symmetry. 
To be more precise, we introduce the following notation.
Suppose that for each $n\in\N$, we have a collection of closed
intervals $\mathcal{I}_n=\{I_{n,i}\}_i$ with mutually disjoint
interiors and open intervals $\mathcal{J}_n=\{J_{n,i}\subset I_{n,i}\}$ such that each $I_{n+1,i}$ is a subset of some $I_{n,j}$, 
$\bigcup\mathcal{I}_{n+1}=\bigcup\mathcal{I}_n\setminus\bigcup\mathcal{J}_n$ and
that $\sup_{j}|I_{n,j}|\rightarrow0$ as $n\rightarrow\infty$. 
We also assume that
$\bigcup\mathcal{I}_1$ is bounded. 
We refer to $\{\mathcal{I}_n, \mathcal{J}_n\}_n$ as a 
\emph{Cantor construction}. The
resulting \emph{Cantor set} is given by
\[C=C_{{\{\mathcal{I}_n,\mathcal{J}_n\}}}=\bigcap_{n}\bigcup_{i}I_{n,i}.\] 
Given the collections $\mathcal{I}_n$ and $\mathcal{J}_n$ as above, we also
denote $\mathcal{I}=\bigcup_{n}\mathcal{I}_n$ and
$\mathcal{J}=\bigcup_{n}\mathcal{J}_n$. 
If there exists $0<c<1$ so that $c I_{n,i}\bigcap J_{n,i}\neq\emptyset$ for
all $I_{n,i}$, we say that our Cantor construction (and set) is
\emph{nice}\footnote{Geometrically, this only means that if the
  removed holes $J_{n,i}$ are small, then they cannot lie too close to
the boundary of $I_{n,i}$.}. Here $c I_{n,i}$ denotes the interval concentric with $I_{n,i}$ and with length $c|I_{n,i}|$.
Furthermore, given a sequence $0<\alpha_n<1$, we say
that the Cantor set $C=C_{{\{\mathcal{I}_n,\mathcal{J}_n\}}_n}$ is
\emph{$(\alpha_n)$-porous} if $|J_{n,i}|\geq\alpha_n |I_{n,i}|$ for
all $I_{n,i}\in\mathcal{I}_n$ and \emph{$(\alpha_n)$-thick}, if
$|J_{n,i}|\leq\alpha_n |I_{n,i}|$
for all $I_{n,i}$. Finally, $C$ is called 
\emph{$(\alpha_n)$-regular} if $\lambda\alpha_n|I_{n,i}|\le |J_{n,i}|\le\Lambda\alpha_n
|I_{n,i}|$ for all $I_{n,i}$ (here $0<\lambda\le\Lambda<\infty$ are
constants that do not depend on $n$ nor $i$). We underline that these definitions do not refer only to the set $C$ but also to the construction of $C$ via $\{\mathcal{I}_n,\mathcal{J}_n\}_n$.

\begin{rems}
\emph{a)} Using our notation, it is possible that a Cantor set $C$ contains
isolated points as some of the intervals $I_{n,i}$ could be
degenerated. 
We allow this for technical reasons although in most
interesting cases, e.g if $C$ is nice, the set $C$ is a
true Cantor set in the sense that it has
no isolated points.\\
\emph{b)} Observe that in our definitions, we do not impose any
conditions on the number or relative size of the intervals
$I_{n+1,j}\subset I_{n,i}$. Note also that
$I_{n+1,i}\in\mathcal{I}_{n+1}$ does not have 
to be a component of any $I_{n,j}\setminus J_{n,j}$.\\ 
\emph{c)} We formulate our results for Cantor sets, but it is reasonable to speak
about $(\alpha_n)$-porosity and  
$(\alpha_n)$-thickness for general subsets of $\R$ and not only for
the ones obtained from Cantor constructions. Roughly speaking,
$A\subset\R$ is $(\alpha_n)$-porous if it is contained in an
$(\alpha_n)$-porous Cantor set and $(\alpha_n)$-thick, if it contains
an $(\alpha_n)$-thick Cantor sets. See \cite{wu}, and \cite{sw} for
more details. In Section \ref{sec:mpor} we provide a notion of
$(\alpha_n)$-porosity which is useful in any metric space.
\end{rems}

Our main result concerning doubling measures and Cantor sets is the
following theorem. 

\begin{thm}\label{thm:main}
Suppose that $C=C_{\{\mathcal{I}_n, \mathcal{J}_n\}}$ is a nice
Cantor set. Then, for each $0<p<\infty$, there is
$\mu\in\mathcal{D}(\R)$ and $0<\lambda\le\Lambda<\infty$ so that
\begin{align}\label{arvio}
\lambda\left(\frac{|J_{n,i}|}{|I_{n,i}|}\right)^p\le\frac{\mu(J_{n,i})}{\mu(I_{n,i})}\le\Lambda\left(\frac{|J_{n,i}|}{|I_{n,i}|}\right)^p .
\end{align}
for each $I_{n,i}$.
\end{thm}

\begin{rem}
This result is interesting already for the middle
interval Cantor sets $C(\alpha_n)$. After the submission of this paper, we were
informed that for uniform Cantor sets, the result has
been proved independently by Peng and Wen. See \cite{pw} for the
precise formulation of their result.
\end{rem}  

Let us now discuss what can be said about the validity of Theorem
\ref{thm:intro} for the general Cantor sets $C_{\{\mathcal{I}_n,\mathcal{J}_n\}}$. 
Observe that Theorem \ref{thm:intro} includes the following four statements: 
\begin{enumerate}[(I)]
\item\label{a} If $(\alpha_n)\notin\bigcup_{0<p<\infty}\ell^p$, 
then $\mu(C)=0$ for all $\mu\in\mathcal{D}(\R)$.
\item If $(\alpha_n)\in\bigcup_{0<p<\infty}\ell^p$, then there
 is\label{d} 
$\mu\in\mathcal{D}(\R)$ with $\mu(C)>0$.
\item\label{b}
If $(\alpha_n)\in\bigcap_{0<p<\infty}\ell^p$, then 
$\mu(C)>0$ for all $\mu\in\mathcal{D}(\R)$.
\item\label{c}
If $(\alpha_n)\notin\bigcap_{0<p<\infty}\ell^p$, then there is
$\mu\in\mathcal{D}(\R)$ so that $\mu(C)=0$.
\end{enumerate}
The Claim \eqref{a} holds for general $(\alpha_n)$-porous sets
$A\subset\R$ as shown by Wu
\cite[Theorem 1]{wu}. In fact, her result remains
true in all metric spaces. 
We provide a simple proof in Lemma \ref{wuprop}.
The Claim \eqref{b} is a special case of a more general result of
Staples and Ward \cite[Theorem 1.4]{sw}. They proved that if $C\subset\R$ is
$(\alpha_n)$-thick for some $(\alpha_n)\in\bigcap_{0<p<\infty}\ell^p$,
then $C$ is fat. 

Our new results in Section \ref{sec:main} deal with the Claims
\eqref{c} and \eqref{d}. These are the claims whose earlier proofs rely
on the symmetries of $C(\alpha_n)$. We 
show that if  
$\sum_{n=1}^\infty\alpha_{n}^p=\infty$ for some $p>0$, and $C$ is
a nice $(\alpha_n)$-porous Cantor set, then
$C\notin\mathcal{F}$. 
On the other hand, if
there is $p<\infty$ with $\sum_{n=1}^\infty\alpha_{n}^p<\infty$,
and if $C$ is a nice 
$(\alpha_n)$-thick Cantor set, then $C\notin\mathcal{T}$. Putting all
these results together, we arrive at a complete analogue of Theorem
\ref{thm:intro} 
for nice $(\alpha_n)$-regular Cantor sets $C\subset\R$.
The proofs of our results in Section \ref{sec:main} are all based on Theorem
\ref{thm:main}.

In the last part of the paper in Section \ref{atomic}, we discuss
\emph{purely atomic} 
doubling measures. Recall that a measure $\mu$ on a metric space $X$ is called
\emph{purely atomic}, if there is a 
countable set $F\subset X$ so that $\mu(X\setminus F)=0$. Purely
atomic doubling measures have reached some attention recently, see e.g.
\cite{kw}, \cite{wu}, \cite{lww}, and \cite{www}. 

Denote by $F_X$ the set of isolated points of a metric space $X$ and let
$E_X=X\setminus F_X$. If $E_X$ is nowhere dense, it is reasonable to
ask if there are purely atomic doubling measures on $X$ and on the other hand,
what conditions guarantee that all doubling measures on $X$ are purely atomic.
We will
treat these questions for a class of metric spaces obtained by adding the
midpoints of the intervals $J\in\mathcal{J}$ to the Cantor sets
$C=C_{\{\mathcal{I}_n,\mathcal{J}_n\}}$. If $C$ is $(\alpha_n)$-regular,
we will classify in terms of the sequence $(\alpha_n)$, which of the
corresponding midpoint sets carry purely 
atomic doubling measures. We find a characterisation of the same nature for all doubling
measures being purely atomic.
The result, Theorem \ref{thm:atomic}, is analogous to Theorem
\ref{thm:intro}. We will 
also answer two questions on atomic doubling measures posed by Kaufman
and Wu \cite{kw}, and Lou, Wen, and Wu \cite{lww}.

We finish this section with some notation.
By a measure on (a metric space) $X$, we always mean a Borel regular
outer measure, defined 
on all subsets of $X$. If $A\subset X$, we denote by $\mu|_A$ the
restriction of $\mu$ to $A$ given by $\mu|_A(B)=\mu(A\cap B)$ for
$B\subset X$. For an interval $I\subset\R$, we denote by $\partial I$
the set of its endpoints.
We adopt the convention that $0<c<\infty$ always denotes a
constant that only depends on parameters which should be clear from the context. Sometimes we write $c=c(a,\ldots,b)$ to emphasize that $c$
depends only on the values of $a,\ldots,b$. For notational convenience, the exact value of $c$ may vary even inside a given chain of inequalities.
Given a family of numbers $0<A_\alpha,B_\alpha<\infty$, parametrised
by $\alpha$, we denote
$A_\alpha\lesssim B_\alpha$ if there is a constant $c$ so that
$A_\alpha\leq c B_\alpha$ for all $\alpha$. By $A_\alpha\approx
B_\alpha$ we mean that $A_\alpha\lesssim
B_\alpha$ and $B_\alpha\lesssim A_\alpha$.

\section{Results for $(\alpha_n)$-porous and $(\alpha_n)$-thick sets}
\label{sec:main}

Our new results concerning $(\alpha_n)$-porous and $(\alpha_n)$-thick
Cantor sets are based on Theorem \ref{thm:main}.

\begin{thm}\label{thm:por}
Suppose that $C=C_{\{\mathcal{I}_n,\mathcal{J}_n\}}$ is nice and 
$(\alpha_n)$-porous for some
$(\alpha_n)\notin\bigcap_{0<p<\infty}\ell^p$. Then there is
$\mu\in\mathcal{D}(\R)$ with $\mu(C)=0$.  
\end{thm}

\begin{proof}
We may assume that $C\subset[0,1]$.
Choose $p>0$ such that $\sum_{n=1}^\infty\alpha_{n}^p=\infty$. Let $\mu$ be a
doubling measure given by Theorem \ref{thm:main}. Then
$\mu(J_{n,i})\approx\bigl(|J_{n,i}|/|I_{n,i}|\bigr)^p\mu(I_{n,i})$. As
$|J_{n,i}|\geq\alpha_n |I_{n,i}|$,
we get
\begin{equation}\label{eq:smallmass}
\mu(J_{n,i})\gtrsim\alpha_{n}^p\mu(I_{n,i}).
\end{equation}
This gives, for some $c>0$,
\begin{equation*}
\mu\Bigl([0,1]\setminus(\cup\mathcal{J}_1\bigcup\ldots\bigcup\cup\mathcal{J}_{n})\Bigr)\leq(1-c\alpha_{n}^p)\mu\Bigl([0,1]\setminus(\cup\mathcal{J}_1\bigcup\ldots\bigcup\cup\mathcal{J}_{n-1})\Bigr)
\end{equation*}
for all $n\in\N$ and consequently,
\begin{equation*}
\mu(C)=\mu\Bigl([0,1]\setminus\bigcup_{n=1}^\infty\mathcal{J}_n\Bigr)\leq\mu[0,1]\prod_{n=1}^\infty(1-c\alpha_{n}^p)=0,   
\end{equation*}
as $\sum_{n=1}^\infty\alpha_{n}^p=\infty$.
\end{proof}

\begin{thm}\label{thm:thick}
Suppose that $C=C_{\{\mathcal{I}_n,\mathcal{J}_n\}}\subset\R$ is nice
$(\alpha_n)$-thick for some
$(\alpha_n)\in\bigcup_{0<p<\infty}\ell^p$. Then there is $\mu\in\mathcal{D}(\R)$ with $\mu(C)>0$. 
\end{thm}

\begin{proof}
The proof is very similar to the proof of Theorem \ref{thm:por} and
thus we skip the details. The estimate \eqref{eq:smallmass} gets
replaced by $\mu(J_{n,i})\lesssim\alpha_{n}^p\mu(I_{n,i})$ and this
leads to $\mu(C)>0$ when $\sum_{n=1}^\infty\alpha_{n}^p<\infty$. 
\end{proof}

Putting together Theorems \ref{thm:por} and \ref{thm:thick}, and the
results of Wu \cite[Theorem 1]{wu}, and Staples
and Ward \cite[Theorem 1.4]{sw} mentioned earlier,
we get the following classification 
for the thinness and fatness of nice $(\alpha_n)$-regular Cantor sets.   

\begin{cor}\label{thm:reg}
If $C\subset\R$ is a nice $(\alpha_n)$-regular Cantor set, then 
\begin{enumerate}
\item
$C$ is thin if and only if $(\alpha_n)\notin\bigcup_{0<p<\infty}\ell^p$.
\item 
$C$ is fat if and only if $(\alpha_n)\in\bigcap_{0<p<\infty}\ell^p$.
\end{enumerate}
\end{cor}

\section{Proof of Theorem \ref{thm:main}}\label{proof}

We begin with the following simple lemma.

\begin{lemma}\label{lemma:regdoubl}
For all $0<p<\infty$ there is $c=c(p)<\infty$ and a doubling
measure $\mu$ on $[0,1]$ such that 
\[c^{-1}t^p\leq
\mu[0,t]=\mu[1-t,1]\leq c t^p\] 
for all $0<t<1$. 
\end{lemma}

\begin{proof}
We obey the following construction. Let $m$ be an integer so
large that $2^{-mp+1}<1$. Define
\[\mu[0,2^{-m}]=\mu[1-2^{-m},1]=2^{-mp}\] 
and let $\mu$ be uniformly
distributed on $[2^{-m},1-2^{-m}]$ with total measure
$1-2^{-mp+1}$. For each integer $k\geq 2$, put
\[\mu[0,2^{-km}]=\mu[1-2^{-km}]=2^{-kmp}\] 
and let $\mu$ be uniformly
distributed on the interval $[2^{-km},2^{-(k-1)m}]$
(resp. $[1-2^{-(k-1)m},1-2^{-km}]$) with total measure $2^{-(k-1)mp}-2^{-kmp}$.
It is now easy to see that $\mu$ has the required properties.
\end{proof}

We now start to prove Theorem \ref{thm:main}. 
We assume without loss of generality that $\inf C=0$ and $\sup C=1$.
Fix $0<p<\infty$ and let $\widetilde{c}<1$ be a constant so that 
\begin{equation}\label{eq:nice}
\widetilde{c} I_{n,i}\cap J_{n,i}\neq\emptyset
\end{equation}
for all $I_{n,i}\in\mathcal{I}$ (such a constant exists since the
Cantor construction is nice). From now on, in this proof, all constants of
comparability will only depend on $p$ and $\widetilde{c}$.

Let $\eta>0$
be a small constant so that 
$\eta^p<\tfrac14$.
 We start by dividing the interval $[0,1]$ into \emph{construction
   intervals}\footnote{This refers to the construction of the measure
   rather than construction of the set $C$.}
of level $1$ and \emph{gaps} of level $1$ as follows. 
For all integers $k\geq2$, we choose gaps $J,J'\in\mathcal{J}$
so that $J\cap[2^{-k},2^{-k+1}]\neq\emptyset$ and
$J'\cap[1-2^{-k+1},1-2^{-k}]\neq\emptyset$. Denote the union of all
these gaps by $\mathcal{G}^{i}_1$.
Let also
$\mathcal{G}^{b}_1=\{J_{n,i}\,:\,I_{n,i}\cap\{0,1\}\neq\emptyset\}$
and $\mathcal{G}_1=\mathcal{G}^{b}_1\cup\mathcal{G}^{i}_1$. 
Call the
elements of $\mathcal{G}_1$ gaps of level one and their complementary
intervals the construction intervals of level one. Denote the
collection of all construction intervals of level one by $\mathcal{C}_1$. 

Next we describe how the total measure $\mu[0,1]=1$ is distributed
among the construction intervals and gaps of level $1$. Denote by
$G^{1}_l$ the rightmost gap 
for which $\dist(G^{1}_l,0)<\eta$ and by $G^{1}_r$ the leftmost gap so that
$\dist(G^{1}_r,1)<\eta$. 
Let $G_1,\ldots,G_n$ be the gaps between $G^{1}_l$ and
$G^{1}_r$ and $K_1,\ldots, K_{n+1}$ the complementary intervals in
between $G^{1}_l,G_1,\ldots,G_n,G^{1}_r$.
It is possible that $G_{l}^1=G_{r}^1$ (if there is a huge
gap in the middle) and in this case, the
collection $\{G_1,\ldots G_n,K_1,\ldots K_{n+1}\}$ is considered to be
empty. 
\begin{claim}\label{lemma:N}
$n\leq c$.
\end{claim}
\begin{proof}[Proof of the Claim]
Clearly, there are at most $-c\log(\eta)$ gaps in
$\mathcal{G}^{i}_1$ whose distance to the boundary of $[0,1]$ is at
least $\eta$ (here and in what follows
$\log=\log_2$). Taking \eqref{eq:nice} into account, we observe that a
similar estimate applies also to the number of elements in
$\mathcal{G}^{b}_1$ whose distance to the boundary of $[0,1]$ is greater
than $\eta$.
\end{proof}

Let $K^{1}_l=[0,\dist(0,G^{1}_l)]$ and $K^{1}_r=[1-\dist(1,G^{1}_r),1]$ and define
\begin{align*}
\mu(K^{1}_l)&=|K^{1}_l|^p, \mu(K^{1}_r)=|K^{1}_r|^p\text{ and}\\
\mu(U)&
=\gamma|U|^p,\text{ for
}U\in\{G^{1}_l,G^{1}_r,G_1,\ldots,G_n,K_1,\ldots,K_{n+1}\} \text{ where}\\
\gamma&
=\frac{1-|K^{1}_l|^p-|K^{1}_r|^p}
{|G^{1}_l|^p+|G^{1}_r|^p+\sum_{i=1}^n|G_i|^p+\sum_{i=1}^{n+1}|K_i|^p}\,.
\end{align*}
(In case $G_l=G_r$, we simply let $\mu(G_l)=1-|K_l|^p-|K_r|^p$.) It follows from the Claim \ref{lemma:N} and the 
choice of $\eta$
that $\tfrac1c\leq\gamma\leq c$ for some $c>1$ and thus
$\mu(U)\approx|U|^p$. We continue distributing the mass inside $K^{1}_l$
(and $K^{1}_r$). Denote by $G_{l}^2$ the rightmost gap inside
$K^{1}_l$ with $\dist(G_{l}^2,0)<\eta|K^1_{l}|$. Define
$K^{2}_l=[0,\dist(0,G^2_{l})]$ and
$\mu(K^2_{l})=(|K^{2}_l|/|K^{1}_l|)^p\mu(K^1_{l})=|K^{2}_l|^p$.
If $U$ is one of the gaps of level $1$ between $G_{l}^2$ and $G_{l}^1$ (resp. $G_{r}^2$
and $G_{r}^1$) or one of the complementary intervals of level $1$ in between these gaps, we
put $\mu(U)=\gamma|U|^p$, where $\gamma$ is a constant defined so that the
total measure of $K^{1}_l$ (resp $K^{1}_r$) remains unchanged. A
similar argument as in the proof of Claim \ref{lemma:N} implies again
that $\gamma\approx 1$.
We continue the construction inductively inside $K^2_{l}$ (and
$K^2_{r}$) by letting $G^{3}_l$ be the rightmost gap inside
$K^{2}_l$ for which $\dist(0,G^{3}_l)<\eta|K^{2}_l|$,
$K^{3}_l=[0,\dist(0, G^{3}_l)]$,
$\mu(K^{3}_l)=(|K^{3}_l|/|K^{2}_l|)^p\mu(K^{2}_l)=|K^{3}_l|^p$ and so
on. Continuing in this 
manner, we eventually get to define the measure of each gap and
construction interval of level one. See Figure \ref{fig:construction}.
\begin{figure}
\centering
\psfrag{a}{$G^{2}_l$}
\psfrag{b}{$G^{1}_l$}
\psfrag{c}{$G^{1}_r$}
\psfrag{d}{$G^{2}_r$}
\psfrag{1}{$K_1$}
\psfrag{3}{$K_2$}
\psfrag{5}{$K_3$}
\psfrag{2}{$G_1$}
\psfrag{4}{$G_2$}
\includegraphics
{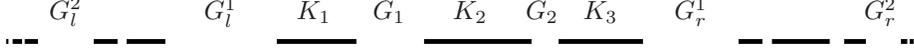}
\caption{The gaps and construction intervals of level one.}
\label{fig:construction}
\end{figure}

We proceed with the mass distribution process inside the construction
intervals of level one. For such an interval $I$, we consider
gaps
$\mathcal{G}^{b}_I=\{J_{n,i}\subset
I\,:\,I_{n,i}\cap\partial I\neq\emptyset\}$ and also let
$\mathcal{G}^{i}_I$ consist of a dyadic 
  sequence of gaps defined similarly as
  $\mathcal{G}^{i}_{[0,1]}=\mathcal{G}^{i}_1$ was defined for $I=[0,1]$. More precisely, if $I=[a,b]$, for each $k\ge 2$ we choose gaps $J,J'\in\mathcal{J}$ so that $J\cap[a+2^{-k}(b-a),a+2^{-k+1}(b-a)]\neq\emptyset$ and $J'\cap[b-2^{-k+1}(b-a), b-2^{-k}(
b-a)]\neq\emptyset$.
  Put $\mathcal{G}_I=\mathcal{G}^{i}_I\cup\mathcal{G}^{b}_I$. We call
  the elements of $\mathcal{G}_I$ the gaps of $I$. Their
  complementary intervals inside $I$ are called the sub-construction intervals
  of $I$. 
  The mass
  $\mu(I)$ is distributed for the gaps and construction intervals of
  level two inside $I$ by the same procedure as the unit mass
was distributed for the gaps and construction
  intervals of level one. The only difference is, that we replace
  $1=\mu[0,1]$ by $\mu(I)$. We repeat this process inductively for all
  construction intervals of all levels. We denote by  $\mathcal{G}_n$ the set of all gaps of level $n$ and by $\mathcal{C}_n$ the collection of construction intervals of
  level $n$. Observe that the construction intervals do not have to be
  covering intervals (i.e. members of $\mathcal{I}$). So most likely, $\mathcal{C}_n\neq\mathcal{I}_n$
  and also
  $\mathcal{G}_n\neq\mathcal{J}_n$ even though
  $\bigcup_n\mathcal{G}_n=\bigcup_{n}\mathcal{J}_n=\mathcal{J}$. Let us further
  denote
  $\mathcal{C}=\bigcup_{n=1}^\infty\mathcal{C}_n$. 

We have now defined the measure of
  all the gaps and construction intervals
and we may use a standard mass distribution principle, see
e.g. \cite[Proposition 1.7]{fa}, 
to define the measure $\mu|_C$.
Inside the gaps the measure will be distributed in the following manner: Let
$G=]a,b[\in\mathcal{J}$. Then we let $\mu|_{G}$ be a doubling measure on $G$
given by a scaled version of Lemma \ref{lemma:regdoubl} so that
\begin{equation}\label{eq:mugap}
\mu]a,a+t]=\mu[b-t,b[\approx \left(\frac{t}{|G|}\right)^p\mu(G)
\end{equation} 
for all $0<t<b-a$. By the proof of Lemma \ref{lemma:regdoubl}, this may be
done in such a way
that the doubling constant of $\mu|_G$ is independent of $G\in\mathcal{J}$. 
This completes the construction of $\mu$. To complete the proof of
Theorem \ref{thm:main}, we
have to show that $\mu$ 
is doubling and satisfies \eqref{arvio}.

Our next claim follows directly from the way $\mu$ is defined.
\begin{claim}\label{aku}
Let $K\in\mathcal{C}_n$ and $I\subset K$,
$I\in\mathcal{C}_{n+1}\cup\mathcal{G}_{n+1}$. Then
\[\mu(I)\approx\left(\frac{|I|}{|K|}\right)^p\mu(K).\] 
If $K=[a,b]$, then for all $0<t<1$,
\[\mu[a,a+t(b-a)]\approx t^p\mu(K)\approx \mu[b-t(b-a)),b]\,.\] 
\end{claim}

Denote $N=\{0,1\}\cup\bigcup_{G\in\mathcal{J}}\partial G$. 
\begin{claim}\label{claim:reg}
Suppose that $I\subset[0,1]$ is an interval with $I\cap
N\neq\emptyset$ and let $K\in\mathcal{C}$ be the
shortest construction interval containing $I$. Then
\[\mu(I)\approx\left(\frac{|I|}{|K|}\right)^p\mu(K).\] 
\end{claim}

\begin{proof}[Proof of claim \ref{claim:reg}]
Denote $K=[a,b]$ and let $c>1$ be a constant from the Claim
\ref{aku} so that
\begin{equation}\label{cdef}
\frac{t^p\mu(K)}{c|K|^p}\leq \mu]a,a+t],\mu[b-t,b[\,\leq \frac{c t^p \mu(K)}{|K|^p}
\end{equation}
for all $0<t<|K|$.
Fix $\varepsilon=\varepsilon(p)>0$ so that $\varepsilon^p\leq1/(2 c^2)$.

Assume first, that $\dist(I,\{a,b\})\geq \varepsilon|I|$. Let
$K_1,\ldots, K_n$ be the sub-construction intervals of $K$
intersecting $I$ and $G_1,\ldots, G_{m}$ ($m\in\{n-1,n,n+1\}$) the
gaps of $K$ intersecting $I$. It may happen that $U\setminus I\neq\emptyset$ for
some (but at most two) $U\in\{K_1,\ldots,K_n,
G_1,\ldots, G_m\}$. In this case, we replace $U$
by $U\cap I$ in the calculation below.
As $\dist(I,\{a,b\})\geq\varepsilon|I|$, it follows as in the proof of
Claim \ref{lemma:N}, that $n,m\leq c$. Using Claim \ref{aku} and
\eqref{eq:mugap}, it now follows that 
\begin{align*}
\mu(I)&=\sum_{i=1}^n\mu(K_i)+\sum_{j=1}^m\mu(G_j)
\approx
\sum_{i=1}^n\left(\frac{|K_i|}{|K|}\right)^p\mu(K)\\
&+\sum_{j=1}^m\left(\frac{|G_j|}{|K|}\right)^p\mu(K)\approx \left(\frac{|I|}{|K|}\right)^p\mu(K).
\end{align*}

Suppose then that $\delta=\dist(I,\{a,b\})<\varepsilon |I|$. We may
assume by symmetry, that $\dist(a,I)<\varepsilon|I|$. The claimed upper bound now
follows from Claim \ref{aku} since 
\[\mu(I)\leq\mu[a,a+2|I|]\approx
\left(\frac{2|I|}{|K|}\right)^p\mu(K).\] 
For the lower bound, we use \eqref{cdef} to obtain
\begin{align*}
\mu(I)&=\mu([a,a+\delta+|I|])-\mu([a,a+\delta])
\geq\frac{\mu(K)}{|K|^p}\left(\frac1c(\delta+|I|)^p-c\delta^p\right)\\ 
&\geq\bigl(|I|/|K|\bigr)^p\mu(K)\left(\tfrac1c-c\varepsilon^p\right)\gtrsim\bigl(|I|/|K|\bigr)^p\mu(K).
\end{align*} 
where the last estimate follows from the choice of $\varepsilon$.
\end{proof}

Now we are ready to verify \eqref{arvio}.
Fix $J=J_{n,i}\in\mathcal{J}$, and
let $K$ be the smallest construction interval containing
$I=I_{n,i}$. By Claim \ref{claim:reg}, we have 
\begin{equation}\label{eq:hassu}
\mu(I)\approx
\left(\frac{|I|}{|K|}\right)^p\mu(K).
\end{equation}
If $J$ is a gap of $K$, it follows from Claim \ref{aku} that
$\mu(J)\approx\bigl(|J|/|K|\bigr)^p\mu(K)$. Combining this with \eqref{eq:hassu},
we get
$\mu(J)/\mu(I)\approx\bigl(|J|/|I|\bigr)^p$. 
If $J$ is not a gap of $K$, we argue as follows: Since $K$ is the smallest construction
interval containing $I$, there is a gap of $K$ intersecting $I$. Thus,
if $K'$ is the sub-construction interval of $K$ containing $J$, we
have $I\cap\partial K'\neq\emptyset$ and consequently $J\in
\mathcal{G}^b_{K'}$. Now, using Claim \ref{aku}, we obtain 
\[\mu(J)\approx \left(\frac{|J|}{|K'|}\right)^p\mu(K')\approx
\left(\frac{|J|}{|K'|}\right)^p\left(\frac{|K'|}{|K|}\right)^p\mu(K)=\left(\frac{|J|}{|K|}\right)^p\mu(K)\]
and it follows
as above that 
$\mu(J)/\mu(I)\approx(|J|/|I|)^p$.
Whence, \eqref{arvio} follows.

It remains to show that 
$\mu$ is doubling on $[0,1]$. For this, it is clearly enough to show that
\begin{equation}\label{eq:adj}
\mu(I_1)\approx\mu(I_2)
\end{equation}
if $I_1$ and $I_2$ are closed sub-intervals of $[0,1]$ with
equal length and $I_1\cap I_2\neq\emptyset$. 
Let $I_1$ and $I_2$ be such intervals aligned from left to right.
If $(I_1\cup I_2)\cap N=\emptyset$, then $I_1\cup I_2\subset G$ for
some $G\in\mathcal{G}$ and
\eqref{eq:adj} follows from the way $\mu|_G$ was defined.

Suppose next that $I_1\cap N\neq\emptyset\neq I_2\cap N$.
Let $K_1$ (resp. $K_2$) be the smallest construction
interval containing $I_1$ (resp. $I_2$). Then $K_1\subset K_2$ or
$K_2\subset K_1$ since $I_1\cap I_2\neq\emptyset$ and any two
construction intervals are either disjoint or within each other. We may assume
without loss of generality, that $K_1\subset K_2$.

\begin{claim}\label{levelsmall}
If $K_1\in\mathcal{C}_n$ and $K_2\in\mathcal{C}_m$, then
$n\leq m+3$.
\end{claim}
\begin{figure}
\centering
\psfrag{a}{$a$}
\psfrag{b}{$b$}
\psfrag{c}{$c$}
\psfrag{1}{$I_1$}
\psfrag{2}{$I_2$}
\psfrag{3}{$K'$}
\psfrag{4}{$G_1$}
\psfrag{5}{$G_2$}
\definecolor{gray1}{rgb}{0.44,0.44,0.44} 
\psfrag{6}{\color{gray1}$G$}
\definecolor{gray2}{rgb}{0.65,0.65,0.65} 
\psfrag{K}{\color{gray2}$K$}
\includegraphics
{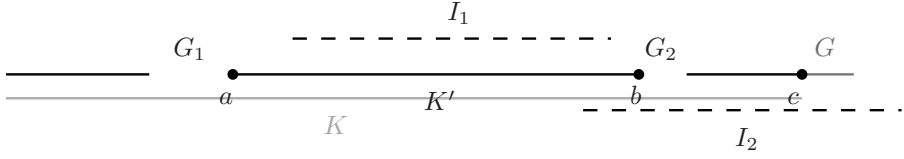}
\caption{Illustration for the proof of Claim \ref{levelsmall}.}
\label{fig:levels}
\end{figure}
\begin{proof}[Proof of claim \ref{levelsmall}]
If $K_1=K_2$, we are done, so assume $K_2\setminus K_1\neq\emptyset$.
Let $G$ be the leftmost gap of $K_2$ that intersects $I_2$ and
let $K\in\mathcal{C}_{m+1}$ be the construction interval next to
$G$ on the lefthand side. As $I_1\subset K_2$ and $G\cap
I_1=\emptyset$ (otherwise $K_1=K_2$), the intervals $K$ and $G$ are
well defined. Moreover, we have $I_1\subset K$.
Consider the collection $\mathcal{G}_K$. If
$I_1\cap\cup\mathcal{G}_K\neq\emptyset$, it follows that $K_1=K$
(i.e. $n=m+1$)
and we are done. Otherwise, there are two consecutive gaps
$G_1,G_2\in\mathcal{G}_K$ and a sub-construction interval of $K$
denoted by $K'\in\mathcal{C}_{m+2}$ in between $G_1$ and $G_2$ so that
$I_1\subset K'$. Let us denote by $a$ the right endpoint 
of $G_1$ ($=$ left endpoint of $K'$), by $b$ the left endpoint of $G_2$
($=$ right endpoint of $K'$), and by $c$ the right endpoint of $K$
($=$ left endpoint of $G$), see Figure \ref{fig:levels}. 
From the way $\mathcal{G}^{i}_{K}$ is constructed, it follows that
$|a-c|<4|b-c|$ and so 
\[|K'|=|a-b|<3|b-c|\leq 3|I_2|=3|I_1|.\] We also know, by
considering $\mathcal{G}^{i}_{K'}$ that all sub-construction intervals
of $K'$ have length at most $|K'|/2$ and similarly their
sub-construction intervals are shorter than $|K'|/4<\tfrac34|I_1|$.
Thus, $|I_1|$ cannot be contained in a construction interval of level
$m+4$ and the claim follows.
\end{proof}

The Claims
\ref{aku} and \ref{levelsmall} now imply that
$\mu(K_1)\approx\bigl(|K_1|/|K_2|\bigr)^p\mu(K_2)$. On the other hand, by Claim 
\ref{claim:reg}, we have $\mu(I_1)\approx\bigl(|I_1|/|K_1|\bigr)^p\mu(K_1)$ as
well as $\mu(I_2)\approx\bigl(|I_2|/|K_2|\bigr)^p\mu(K_2)$. Putting these
estimates together implies
\begin{align*}
\mu(I_1)&\approx\left(\frac{|I_1|}{|K_1|}\right)^p\mu(K_1)
\approx\left(\frac{|I_1|}{|K_1|}\right)^p\left(\frac{|K_1|}{|K_2|}\right)^p\mu(K_2)\\
&=\left(\frac{|I_2|}{|K_2|}\right)^p\mu(K_2)\approx\mu(I_2).
\end{align*}

Suppose finally, that only one
of the
intervals $I_1$ or $I_2$, say $I_2$, hits $N$. Then $I_1$ is a subset
of a gap $G=]a,b[$ and $\delta=\dist(I_1,b)\leq|I_1|$. Letting
$I_{3}=I_1+\delta$, we have
$\mu(I_1)\approx\mu(I_{3})$ as $\mu|_G$ is doubling. On the other
hand, since $I_{3}\cap N\neq\emptyset\neq I_{2}\cap N$, and $I_2\cap
I_3\neq\emptyset$, we already know that
$\mu(I_{3})\approx\mu(I_{2})$. 
Combining these estimates, we get $\mu(I_1)\approx\mu(I_2)$.
This completes the proof of Theorem \ref{thm:main} 
\hfill$\Box$

\smallskip

It is natural to ask if we could drop the word ``nice'' from the
assumptions in Theorems
\ref{thm:main}, \ref{thm:por}, and \ref{thm:thick} and in Corollary \ref{thm:reg}.  
The Proposition below shows (by choosing a fat Cantor set as $E$) that, at least, in Theorems \ref{thm:main}
and \ref{thm:por} this is not possible. We do not know if one could
remove this assumption from Theorem \ref{thm:thick}. 

\begin{prop}
If $E\subset\R$ is nowhere dense and $0<p<1$, then there is a Cantor set
$C\supset E$ which is
$(\alpha_n)$-porous for some $(\alpha_n)\notin\ell^p$.
\end{prop}

\begin{proof}
We construct inductively the required intervals $I_{n,i}$ and
$J_{n,i}$ that satisfy $E\subset[0,1]\setminus\bigcup_{n,i}J_{n,i}$. 

\textbf{Step 1:}
Pick any subinterval $G\subset[0,1]\setminus E$ of length
$\leq\tfrac12$ so that $G\cap[\tfrac14,\tfrac34]\neq\emptyset$ and
denote $r=|G|$. Choose a number $M_1\in\N$ so that 
\begin{equation}\label{Mlarge}
M_{1}^{1-p}(r/2)^p\geq 1.
\end{equation}
Let $J_1,\ldots J_{2M_1}$ be disjoint open sub-intervals of $G$ with length $\delta=r/(2M_1)$,
 enumerated from left to right. Define 
$\alpha_1=\alpha_2=\ldots\alpha_{M_1}=\delta$.  
From \eqref{Mlarge}, we get
$\sum_{n=1}^{M_1}\alpha_{n}^p\geq 1$.
If $a$ is the centre point of
$G$, define $\mathcal{I}_1=\{[0,a], [a,1]\}$
$\mathcal{J}_1=\{J_{M_1},J_{M_1+1}\}$, $\mathcal{I}_2=\{[0,a-\delta],
[a+\delta,1]\}$, $\mathcal{J}_2=\{J_{M_1-1}, J_{M_1+2}\}$,\ldots,
$\mathcal{I}_{M_1}=\{[0,a-r/2+\delta], [a+r/2-\delta,1]\}$,
$\mathcal{J}_{M_1}=\{J_1, J_{2M_1}\}$.  

\textbf{Step m:} Suppose that $M_1,\ldots,M_{m-1}\in\N$ as well as
the collections $\mathcal{I}_j$, $\mathcal{J}_j$ for $1\leq j\leq
\sum_{k=1}^{m-1}M_k$ have been defined.
We now perform the step 1 construction inside each
of the elements of $\mathcal{I}_{\sum_{k=1}^{m-1}M_k}$. The number $M_m$
  as well as $\alpha_n$ for $\sum_{k=1}^{m-1}M_k<
  n\leq\sum_{k=1}^{m}M_k$ will be determined according to the smallest
  relative gap chosen inside the intervals
  $I\in\mathcal{I}_{\sum_{k=1}^{m-1}M_k}$, and we choose the number $M_m$
  so large, that
\[\sum_{n=\sum_{k=1}^{m-1}M_k}^{\sum_{k=1}^{m}M_k}\alpha_{n}^p\geq
1.\]
It is now evident from the construction, that $(\alpha_n)\notin\ell^p$
and that the set $C=\bigcap_{j=1}^\infty\bigcup\mathcal{I}_j$ is
$(\alpha_n)$-porous.    
\end{proof}

\begin{rem}
To formally fulfill the requirement
$\bigcup\mathcal{I}_{n+1}=\bigcup\mathcal{I}_n\setminus\bigcup\mathcal{J}_n$ we
should add to each $\mathcal{I}_{n+1}$ the 
boundary points of the deleted intervals $J\in\mathcal{J}_{n}$
and also emptysets as their "holes" to $\mathcal{J}_{n+1}$. 
For those readers who consider this cheating, we suggest to modify the
construction so that
$\bigcup\mathcal{I}_{n+1}=\bigcup\mathcal{I}_n\setminus\bigcup\mathcal{J}_n$
holds and the resulting Cantor set
$C=C_{\{\mathcal{I}_n,\mathcal{J}_n\}_n}$ contains no isolated
points.
It is also possible to
modify the construction so that $(\alpha_n)\notin\bigcup_{0<q<1}\ell^q$. 
\end{rem}

\section{A lemma on $(\alpha_n)$-porous sets in metric spaces}
\label{sec:mpor}

For the purpose of proving results for midpoint Cantor sets in
Section \ref{atomic}, 
we present here a metric space version of Wu's result on
$(\alpha_n)$-porous sets being 
null for all doubling measures if
$(\alpha_n)\notin\bigcup_{0<p<\infty}\ell^p$. Her argument to prove the
result in $\R$ readily works in much general situations once we find a
reasonable definition of $(\alpha_n)$-porosity to use.
There are basically two
options: If one wants that the covering collection consists of
distinct elements, then one has to use more general covering objects
than just balls or intervals. The second option, which is more useful for us, is
to relax the disjointness condition a bit and still keep using coverings
with balls. For an analogous result using the first mentioned option, see
\cite[Theorem 4.9]{l}.

We say that a subset $E\subset X$ of a metric space $X$ is 
\textit{$(\alpha_n)$-porous} for a sequence $(\alpha_n)_{n=1}^\infty$, 
$0<\alpha_n<1$, if 
there is a constant $N\in\N$ and a sequence of (finite or countably
infinite) coverings
$\mathcal{B}_n=\{B_{n,j}\}_i$ of $E$ by balls
$B_{n,j}=B(x_{n,j},r_{n,j})$ with the following properties: 
\begin{enumerate}[(P1)]
\item\label{ap1}Each 
$B_{n,j}$ contains a sub ball $B'_{n,j}=B(y_{n,j},\alpha_n r_{n,j})\subset B_{n,j}\setminus E$. 
\item\label{ap2} 
Each point $x\in X$ belongs to at most $N$ different balls $B'_{n,j}$.
\end{enumerate} 
It is clear that if $C=C_{\{\mathcal{I}_n,\mathcal{J}_n\}}\subset\R$ is
$(\alpha_n)$-porous in the sense defined in the introduction, then it is
also $(\alpha_n)$-porous in the sense of the above definition.

\begin{lemma}\label{wuprop}
Let $X$ be a metric space.
If $(\alpha_n)\notin\bigcup_{0<p<\infty}\ell^p$, and $E\subset X$ is
$(\alpha_n)$-porous, 
then $\mu(E)=0$ for all doubling measures $\mu$ on $X$.
\end{lemma}
\begin{proof}

Let
$\mathcal{B}_n$ be the coverings that fulfill the
$(\alpha_n)$-porosity conditions (P\ref{ap1}) and (P\ref{ap2}) and
let $\mu$ be a doubling measure on $X$ with doubling constant $1<c<\infty$. 
Without loss of generality, we may assume that $E$
is bounded and that $B_{n,j}\subset B$ for some fixed ball $B\subset X$.
For each $n$, let $k_n$ be the smallest integer so that
$k_n\geq-\log(\alpha_n)+1$. Then 
$B_{n,i}\subset B(y_{n,j},2^{k_n}\alpha_n r_{n,j})$ 
 for all $B_{n,i}\in\mathcal{B}_n$ and thus the doubling condition
gives
\begin{equation*}
\mu(E)\leq\sum_i\mu(B_{n,i})\leq
c^{-\log(\alpha_n)+1}\sum_i\mu(B'_{n,i})=c\,\alpha_{n}^{-p}\sum_{i}\mu(B'_{n,i}),
\end{equation*}
 where $p=\log c>0$. 
Let $\varepsilon>0$. To complete the proof it suffices to find
$n\in\N$ so that
$\sum_{i}\mu(B'_{n,i})\leq
\varepsilon\alpha_{n}^{p}$.
But if this is not the case, then
(P\ref{ap2}) yields
\[\infty>\mu(B)\geq\frac{1}{N}\sum_{n}\sum_i\mu(B'_{n,i})
>\varepsilon\sum_{n=1}^\infty\alpha_{n}^{p}=\infty\]   
giving a contradiction.
\end{proof}

\section{Purely atomic doubling measures}

\label{atomic}

\subsection{On midpoint Cantor sets}

In this subsection, we show how the theorems of Section \ref{sec:main}
can be turned to theorems on atomic doubling measures for certain class of
midpoint Cantor sets.

For each Cantor set $C=C_{\{\mathcal{I}_n,\mathcal{J}_n\}}$, we
define a
\textit{midpoint Cantor set} $M=M_{\{\mathcal{I}_n,\mathcal{J}_n\}}$ by
letting $M=C\cup_{J\in\mathcal{J}}\{x_J\}$, where $x_J$ is the centre point
of $J\in\mathcal{J}$. If $C$ is a middle interval Cantor set $C=C(\alpha_n)$, we
denote the corresponding midpoint Cantor set by $M(\alpha_n)$. We consider
each such $M$ as a metric space, with the inherited Euclidean metric. 

For these midpoint Cantor sets, we verify the following results analogous
to the results obtained for doubling measures on the real-line.
\begin{thm}\label{thm:atomic}
Suppose that $C=C_{\{\mathcal{I}_n,\mathcal{J}_n\}}$ is a Cantor set
and let
$M=M_{\{\mathcal{I}_n,\mathcal{J}_n\}}$. Then:
\begin{enumerate}
\item If $C$ is $(\alpha_n)$-porous for some
  $(\alpha_n)\notin\bigcup_{0<p<\infty}\ell^p$, then all doubling
  measures on $M$ are 
  purely atomic.\label{apor}
\newcounter{enumi_saved}
\setcounter{enumi_saved}{\value{enumi}}
\end{enumerate}
Suppose further that $C$ is nice and let $c$ be a constant so that $J_{n,i}\cap cI_{n,i}\neq\emptyset$
for all $I_{n,i}$. If
\begin{equation}\label{alphansmall}
|J_{n,i}|<\frac{1-c}{3}|I_{n,i}|\text{ for each }I_{n,i},
\end{equation}
then also the following holds:
\begin{enumerate}
\setcounter{enumi}{\value{enumi_saved}}
\item If $C$ is $(\alpha_n)$-thick for some
  $(\alpha_n)\in\bigcup_{0<p<\infty}\ell^p$, then there are doubling measures $\mu$
  on $M$ with $\mu(C)>0$.\label{athick'}
\item If $C$ is $(\alpha_n)$-thick for some
  $(\alpha_n)\in\bigcap_{0<p<\infty}\ell^p$, then there are no purely
  atomic doubling measures on $M$.\label{athick} 
\item If $C$ is $(\alpha_n)$-porous and
  $(\alpha_n)\notin\bigcap_{0<p<\infty}\ell^p$, then there are purely
  atomic doubling measures on $M$. \label{apor'} 
\item Finally, suppose that $C$ is nice and $(\alpha_n)$-regular. Then all
  doubling measures on $M$ are purely atomic if and only if
  $(\alpha_n)\notin\bigcup_{0<p<\infty}\ell^p$. There are no purely
  atomic doubling measures on $M$ if and only if
  $(\alpha_n)\in\bigcap_{0<p<\infty}\ell^p$.\label{areg} 
\end{enumerate}
\end{thm}

Our main tool to prove Theorem \ref{thm:atomic} is the following
lemma. We denote by $\delta_x$ the Dirac unit mass located at $x\in\R$.

\begin{lemma}\label{lemma:exp}
Suppose that $C=C_{\{\mathcal{I}_n,\mathcal{J}_n\}}$ is a nice Cantor set
and assume that \eqref{alphansmall} holds.
Let $M=M_{\{\mathcal{I}_n,\mathcal{J}_n\}}$.  
If $\mu$ is a doubling measure on $[\inf C,\sup C]$, we may define a doubling
measure $\nu$ on $M$ by setting
$\nu=\mu|_C+\sum_{J\in\mathcal{J}}\mu(J)\delta_{x_J}$. On the other
hand, if $\nu$ is a doubling measure on $M$, there is a doubling
measure $\mu$ on $[\inf C,\sup C]$ so that $\nu|_C=\mu|_C$ and
$\mu(J)=\nu\{x_J\}$ for all $J\in\mathcal{J}$. 
\end{lemma}

Before starting to prove Lemma \ref{lemma:exp}, we state a couple of
auxiliary results. The first one is a direct 
consequence of the doubling property.

\begin{lemma}\label{lemma:basic}
Let $\mu$ be a doubling measure on a metric space $X$ and let
$1<\Lambda<\infty$. Suppose that $x,y\in X$, $d(x,y)\leq \Lambda r$, and
$1/\Lambda\leq r/s\leq \Lambda$. Then $\mu(B(x,r))\approx\mu(B(y,s))$ where the
constants of comparability only depend on $\Lambda$ and the doubling
constant of $\mu$.
\end{lemma} 

\begin{lemma}\label{apulemma}
Under the assumptions of Lemma \ref{lemma:exp}, there is $c>0$ so that
the following holds: If 
$J,J'\in\mathcal{J}$ and $K$ is the interval between $J$ and $J'$,
then $|K|\geq c \min\{|J|, |J'|\}$.  
\end{lemma}
\begin{proof}
Let $I$ (resp. $I'$) be the smallest interval from $\mathcal{I}$
containing $J$ (resp. $J'$).
Then $I\subset I'$, $I'\subset I$ or $I\cap
I'=\emptyset$. In any case, $J\cap I'=\emptyset$ or $J'\cap
I=\emptyset$. We may assume that $J'\cap
I=\emptyset$. 
Using
\eqref{alphansmall}, we get
$|K|=\dist(J,J')\geq\dist(J,\partial
I)\geq(1-c)|I|/2-|J|\geq|I|(1-c)/6>|J|(1-c)/6$.  
\end{proof}

\begin{proof}[Proof of Lemma \ref{lemma:exp}]
We assume without loss of generality that $\inf C=0$, $\sup C=1$.
By $B(x,r)$ we denote the Euclidean interval $B(x,r)=]x-r,x+r[$
whereas $B_M(x,r)=B(x,r)\cap M$, for $x\in M$.

To prove the first assertion, suppose that $\mu$ is a doubling
measure on $[0,1]$ and let $\nu$ be defined as in the lemma.  
We have to verify that $\nu$ is a doubling measure on $M$. Fix $x\in M$ and
$r>0$. If $B(x,2r)\cap C=\emptyset$, we have
$B_M(x,r)=B_M(x,2r)=\{x\}$ and there is nothing to prove. Proving that
$\nu$ is doubling thus reduces to showing the following. If $x\in M$ and
$B(x,2r)\cap C\neq \emptyset$, then
\begin{align}
\label{tupla}\nu(B_M(x,2r))&\lesssim\mu(B(x,r))\text{ and}\\
\label{single}\nu(B_M(x,r))&\gtrsim\mu(B(x,r)).
\end{align}
We may write $\nu(B_M(x,2r))=\mu[a,b]+\nu(E)$, where
$a=\inf(B(x,2r)\cap C)$, $b=\sup(B(x,2r)\cap C)$ and $E$ is either
empty or contains one or two isolated points of $M$. By the
construction of $\nu$,
we have $\nu(E)\leq\mu(B(x,4r))$ and thus
\begin{align*}
&\nu(B_M(x,2r))\leq\mu[a,b]+\mu(B(x,4r))\leq
\mu(B(x,2r))+\mu(B(x,4r))\\
&\lesssim\mu(B(x,r))
\end{align*} 
since $\mu$ is
doubling. Thus \eqref{tupla} follows. 
  
To show \eqref{single}, assume first that $B(x,r/2)\cap
C\neq\emptyset$. If we let $a=\inf(B(x,r)\cap C)$, $b=\sup(B(x,r)\cap
C)$, then Lemma \ref{apulemma} implies $|b-a|\gtrsim r$ and thus
\[\nu(B_M(x,r))\geq\nu([a,b]\cap M)=\mu[a,b]\gtrsim\mu(B(x,r))\] 
by
Lemma \ref{lemma:basic}. If
$B(x,r/2)\cap C=\emptyset$, then $B(x,r/2)\subset J$ for some
$J\in\mathcal{J}$ with $|J|\geq r$, and we have
\[\nu(B_M(x,r))\geq \nu\{x\}=\mu(J)\gtrsim\mu(B(x,r)).\] 
Thus we have
\eqref{single} and it follows that $\nu$ is a doubling measure on $M$.

To give the details for the latter claim of the lemma requires a bit more work. Consider a doubling measure $\nu$ on $M$. We define $\mu$ by the following
procedure: Let $c>0$ be the constant of Lemma \ref{apulemma} and
choose $1/(1+c)<t<1$. For $J=]x-r,x+r[\in\mathcal{J}$, consider
its division to Whitney type sub-intervals
\begin{align*}
J^{+}_{k}&=]x+r-(1-t)^kr,x+r-(1-t)^{k+1}r[,\\
J^{-}_{k}&=]x-r+(1-t)^{k+1}r,x-r+(1-t)^{k}r[
\end{align*}
 for
$k\in\{0,1,2,\ldots\}$. Next define
\begin{align*}
m_{J_{k}^+}&=\nu([x+r+(1-t)^{k+1}r,x+r+(1-t)^kr[\cap M),\\
m_{J_{k}^-}&=\nu(]x-r-(1-t)^{k}r,x-r-(1-t)^{k+1}r]\cap M). 
\end{align*}
If $K$ is
one of the intervals $J^{+}_k$ or $J^{-}_k$, let $\mu|_{K}$  
be uniformly distributed on $K$ with total measure
\[\mu(K)=\frac{m_{K}\nu\{x_J\}}{\nu\left((2J\cap M)\setminus\{x_J\}\right)}.\] 
Observe
that the scaling factor $\nu\{x_J\}/\nu((2J\cap M)\setminus\{x_J\})$
is bounded away form $0$ and $\infty$ as $\nu$ is doubling. Thus we
have 
\begin{equation}\label{mK}
\mu(K)\approx m_{K} 
\end{equation}
for all $K\in\{J_{k}^+, J_{k}^{-}\}_{k=0}^\infty$.
To complete the definition of $\mu$, we
set $\mu|_C=\nu|_C$. 

It is now evident that $\mu(J)=\nu\{x_J\}$ for each $J\in\mathcal{J}$
and it remains to show that $\mu$ is doubling. For this purpose, we prove the following chain of claims.
We formulate some of the claims for $u_2$
and $J^+_{k}$ but due to symmetry, similar 
claims are valid for $u_1$ and $J^{-}_k$ as well.

Let $J=]u_1,u_2[\in\mathcal{J}$. Then 
\begin{enumerate}[(i)]
\item \label{c1} For each $J^{+}_k$, there is $y\in M$ with $y-u_2\approx
  |J^{+}_k|$ such that
$\nu(B_M(y,|J^{+}_k|)\approx\mu(J^{+}_k)$.
\item If $K_0$ and $K_1$\label{c2}
are two consecutive intervals among $J^+_{k}$, $J^{-}_k$, then
$m_{K_0}\approx m_{K_1}$.
\item $\mu(J^{+}_k)\approx\mu(\cup_{n>k}J^{+}_k)$\label{c3}
\item If $0<s<|J|/2$, then
  $\nu([u_2,u_2+s]\cap M)\approx\mu[u_2-s,u_2]$
\label{c4}
\item If $I$ is an interval with $I\cap C\neq\emptyset$ and
  $\kappa>1$, then 
  $\mu(I)\leq c(\kappa,t)\nu(\kappa I\cap M)$.\label{cextra}
\item If $0<s<|J|$, then $\mu[u_2,u_2+s]\approx\nu([u_2,u_2+s]\cap
  M)$.\label{c5}  
\end{enumerate}

We now start to prove the claims \eqref{c1}--\eqref{c5}. Let $c>0$ be
the constant of Lemma \ref{apulemma}. Since $t>1/(1+c)$, 
we may choose $\varrho=\varrho(t)>0$ such that $1-t+\varrho t<c(1-2\varrho)t$.
Let $K=J^{+}_k$.  
By scaling, we may assume that $(1-t)^kr=1$ so that 
$|K|=t$ and $\dist(K,u_2)=1-t$. Denote 
\begin{align*}
K'&=[u_2+(1-t),u_2+1[,\\
(1-2\varrho)K'&=[u_2+(1-t)+\varrho t, u_2+1-\varrho t[. 
\end{align*}
It follows from
Lemma \ref{apulemma} and the choice of $t$ and $\varrho$ that
$(1-2\varrho)K'\cap M\neq\emptyset$. Thus, we may choose $y\in M$ so
that $B_M(y,\varrho t)\subset K'$. See Figure \ref{fig:rho}. Using the
doubling property of 
$\nu$, and the way $\mu$ is defined, we get 
\[\mu(K)\approx\nu(K'\cap
M)\geq\nu(B_M(y,\varrho t))\gtrsim\nu
(B_M(y,t))\geq\nu(K'\cap M)\approx\mu(K).\] 
As $(1-t)\leq|y-u_2|\leq
1$, we have $|y-u_2|\approx t=|J^{+}_k|$ and \eqref{c1}
follows. 

\begin{figure}
\centering
\psfrag{1}{$K$}
\psfrag{2}{$t$}
\psfrag{3}{$1-t$}
\psfrag{4}{$u_2$}
\psfrag{5}{$K'$}
\psfrag{6}{$\varrho t$}
\psfrag{7}{$y$}
\psfrag{8}{$(1-2\varrho)K'$}
\includegraphics
{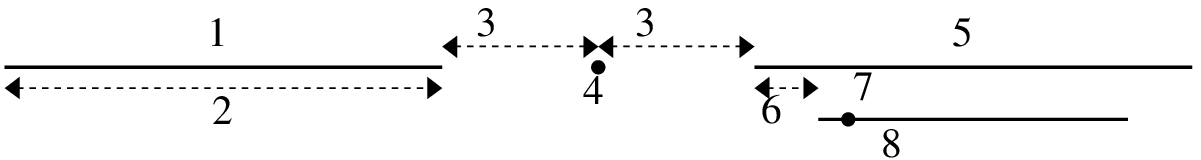}
\caption{Illustration for the proof of \eqref{c1}.}
\label{fig:rho}
\end{figure}

Let $K_0$ and $K_1$ be two consecutive intervals
among $\{J^+_{k}, J^{-}_k\}$, and $y_0, y_1\in M$ be points given
by \eqref{c1}. Then
$|y_0-y_1|\lesssim|K_0|\approx|K_1|$ and combined with \eqref{mK} and Lemma
\ref{lemma:basic}, we get
$m_{K_0}\approx\mu(K_0)\approx\nu(B_M(y_0,|K_0|))\approx\nu(B_M(y_1,|K_1|))
\approx\mu(K_1)\approx
m_{K_1}$ implying \eqref{c2}.

For $k=0,1,2,\ldots$, let $y_k\in M$ be a point satisfying \eqref{c1}. Since
$\nu$ is doubling, we get (using \eqref{mK} and Lemma \ref{lemma:basic})
\begin{gather*}
\mu(\bigcup_{n=k+1}^\infty
J^{+}_n)=\mu[u_2-\tfrac{1-t}{t}|J^{+}_k|,u_2]\approx\sum_{n>k}m_{J^{+}_n}\\
=\nu\left([u_2,u_2+\tfrac{1-t}{t}|J^{+}_k|[\cap
M\right) \lesssim\nu(B_M(y_k,|J^{+}_k|)\approx\mu(J^{+}_k). 
\end{gather*}
On
the other hand, using \eqref{c2} we see that $\mu(J^{+}_k)\approx
m_{J^{+}_k}\approx
m_{J^{+}_{k+1}}\approx\mu(J^{+}_{k+1})\leq\mu(\bigcup_{n=k+1}^\infty
  J^{+}_{n})$ and \eqref{c3} follows.

By construction, we have
\[\mu[u_2-\tfrac1t|J^{+}_k|,u_2]
\approx\sum_{n\geq k}m_{J^{+}_n}
=\nu[u_2,u_2+\tfrac1t|J^{+}_k|[.\] 
Combining this with
\eqref{c3} yields \eqref{c4}.

Let $s<|J|$ and let $K$ be the largest interval among $\{J^{+}_k\}$
contained in $[u_2-s,u_2]$. With the help of \eqref{c1}--\eqref{c3},
we see that 
\[\mu[u_2-s,u_2]\approx\mu(K)\approx m_K\leq\nu([u_2,u_2+s]\cap M)\] (and
similarly $\mu[u_1,u_1+s]\lesssim\nu([u_1-s,u_1]\cap M)$. 
To prove \eqref{cextra}, we apply this
observation for the components 
of $I\setminus C$ to obtain $\mu(I\setminus C)\lesssim\nu(3 I\cap
M)$. Choosing $y\in I\cap C$, we have 
\begin{align*}
\mu(I)&=\nu(I\cap
C)+\mu(I\setminus C)\lesssim 2\nu(3 I\cap M)\lesssim
c(\kappa,t)\,\nu\left(B_M(y,\tfrac{\kappa-1}{2} I)\right)\\
&\leq c(\kappa,t)\nu(\kappa
I\cap M)
\end{align*} 
for each $\kappa>1$. 

To prove \eqref{c5}, let $v=\sup C\cap[u_2,u_2+s]$. Using Lemma
\ref{apulemma}, we may find $y\in M$ and $r\gtrsim s$ so that
$B_M(y,r)\subset[u_2,v]$. Now 
\[\nu([u_2,u_2+s]\cap
M)\lesssim\nu(B_M(y,r))\leq\nu([u_2,v]\cap
M)=\mu[u_2,v]\leq\mu[u_2,u_2+s].\] 
On the other hand, we have
$\mu[u_2,v]=\nu([u_2,v]\cap M)$ and
\[\mu[v,u_2+s]\lesssim\nu([u_2,v+s]\cap M)
\lesssim\nu(B_M(y,r))\leq\nu([u_2,u_2+s]\cap M)\] 
using
\eqref{cextra}. Thus \eqref{c5} follows and we have verified 
all the claims \eqref{c1}--\eqref{c5}.

Let $I_1, I_2\subset [0,1]$ be two adjacent closed intervals of the same length. To
finish the proof we have to show that 
\begin{equation}\label{pakko}
\mu(I_1)\approx\mu(I_2). 
\end{equation}
To achieve this goal, we consider several different cases and
subcases.

\textit{Case a:} Both intervals $I_1$ and $I_2$ are contained in a gap
$J=]u_1,u_2[\in\mathcal{J}$. Let $\mathcal{K}=\{J^{+}_k,
J^{-}_k\,:\,k=0,1,2,\ldots\}$\\ 
\textit{Subcase a1:} If both intervals $I_1$ and $I_2$ intersect at
most $2$ intervals of $\mathcal{K}$, the estimate \eqref{pakko}
follows directly from \eqref{c2}.\\ 
\textit{Subcase a2:} If both intervals $I_i$ intersect at least $3$
elements of $\mathcal{K}$, let $K_i$ be the largest element $K\in\mathcal{K}$
contained in $I_i$. Then, it follows from \eqref{c3} and \eqref{c2}
that $\mu(I_i)\approx\mu(K_i)$. On the other hand, there is at most
one interval $K\in\mathcal{K}$ in between $K_1$ and $K_2$ and thus,
using \eqref{c2} once again, we get $\mu(K_1)\approx\mu(K_2)$.\\ 
\textit{Subcase a3:} Suppose that $I_1$ intersects at least three
sub-intervals $K\in\mathcal{K}$ whereas $I_2$ intersects at most two of
them. Again, letting $K_1$ be the largest element of $\mathcal{K}$
contained in $I_1$, we have $\mu(I_1)\approx\mu(K_1)$. Now, if
$K_2\in\mathcal{K}$ and $K_2\cap I_2\neq\emptyset$, there are at most
two intervals of $\mathcal{K}$ in between $K_1$ and $K_2$. Thus, from
\eqref{c2} we get $m_{K_2}\approx m_{K_1}$ giving
$\mu(I_1)\approx\mu(I_2)$.

\textit{Case b:} $I_1$ is contained in a gap 
but $I_2\cap C\neq\emptyset$.
We may assume by symmetry that $I_1=[a,b]$, $I_2=[b,c]$ (where $c-b=b-a$). 
Let $d=\inf(I_2\cap C)$.\\ 
\textit{Subcase b1:} If $d-b\geq c-d$, the claims
\eqref{c5} and \eqref{c4} imply $\mu[b,c]\approx\mu[b,d]$ and from the case
\textit{a} and Lemma \ref{lemma:basic}, we obtain $\mu[b,d]\approx\mu[a,b]$.\\
\textit{Subcase b2:} If $d-b\leq c-d$, we first use the case \textit{a} to
get $\mu[a,b]\approx\mu[d-(c-d),d]$ and then \eqref{c5} and \eqref{c4}
to conclude
$\mu[d-(c-d),d]\approx\mu[d,c]\approx\mu[b,c]$.

\textit{Case c:} $I_1\cap C\neq\emptyset\neq I_2\cap C$. By symmetry,
we assume again that $I_1=[a,b]$, $I_2=[b,c]$, and denote
$r=b-a=c-b$. Let $v_1=\inf(I_1\cap C)$, $v_2=\sup(I_1\cap C)$,
$v_3=\inf(I_2\cap C)$, and $v_4=\sup(I_2\cap C)$.\\ 
\textit{Subcase c1:} If $v_2-v_1\geq r/2$ and $v_4-v_3\geq r/2$,
we can find $y_1\in M$ so that $B_M(y_1,r/8)\subset
[v_1,v_2]\cap M$ and 
\[\mu(I_1)\geq \mu[v_1,v_2]=\nu([v_1,v_2]\cap
M)\approx\nu(B_M(y_1,r/8)).\] 
As also $\mu(I_1)\lesssim\nu(2 I_1\cap
M)$ by \eqref{cextra}, $2I_1\cap M\subset B_M(y_1,2r)$, and
$\nu$ is doubling, we thus get
$\mu(I_1)\approx\nu(B_M(y_1,r))$. Repeating the argument for $I_2$
yields $B_M(y_2,r/8)\subset[v_3,v_4]\cap M$ with
$\mu(I_2)\approx\nu(B_M(y_2,r))$. Using Lemma \ref{lemma:basic}, we get
$\nu(B_M(y_1,r))\approx\nu(B_M(y_2,r))$ yielding \eqref{pakko}.\\ 
\textit{Subcase c2:} Suppose $v_2-v_1\geq r/2$ and
$v_4-v_3<r/2$. Now, as in subcase \textit{c1}, we find
$B_M(y_1,r/8)\subset [v_1,v_2]\cap M$ with
$\mu(I_1)\approx\nu(B_M(y_1,r))$. On the other hand, letting $I_3$ be
the longer of the intervals $[b,v_3]$ and $[v_4,c]$, with the help of
\eqref{c1}--\eqref{c3}, we find $y_2\in M$ with $\dist(I_3,y_2)\lesssim r$
and $s\approx  r$ such that
$\mu(I_3)\approx\nu(B_M(y_2,s))$.
Again, as $\nu$ is doubling we can use Lemma \ref{lemma:basic} 
to conclude that
$\mu(I_1)\approx\nu(B_M(y_1,r))\approx\nu(B_M(y_2,s))\approx\mu(I_2)$
as desired.\\ 
\textit{Subcase c3:} Finally, 
if both
$v_2-v_1<r/2$ and $v_4-v_3<r/2$, we let $I_3$ be the longer of the
intervals $[a,v_1]$, $[v_2,b]$ and $I_4$ the longer of the
sub-intervals $[b,v_3]$, $[v_4,c]$. As above, we find $B_M(y_1,s_1)$ and
$B_M(y_2,s_2)$ so that
$\mu(I_1)\approx\nu(B_M(y_1,s_1))\approx\nu(B_M(y_2,s_2)\approx\mu(I_2)$.
\end{proof}

\begin{proof}[Proof of Theorem \ref{thm:atomic}]
Suppose first that $C$ is $(\alpha_n)$-porous as a subset of $\R$. The
Claim \eqref{apor} follows from the Lemma \ref{wuprop} since $C$ is
$(\alpha_n/2)$-porous as a subset of $M$. Indeed, for each
$J_{n,i}$, let $x_{n,i}=x_{J_{n,i}}$ and consider
$B_{n,i}=B_M(x_{n,i},|I_{n,i}|)$ and $B'_{n,i}=B_M(x_{n,i},
|J_{n,i}|/2)$. Then $B'_{n,i}\cap B'_{l,j}=\emptyset$ if
$(n,i)\neq(l,j)$ and moreover, $|J_{n,i}|/2\geq(\alpha_n/2)|I_{n,i}|$ for
all $n$ and $i$.

To prove the claims \eqref{athick'}--\eqref{apor'}, we use Lemma \ref{lemma:exp}.
Then \eqref{athick'} follows from
Theorem \ref{thm:thick}, \eqref{athick} from the result of
Staples and Ward \cite[Theorem 1.4]{sw}, and \eqref{apor'}
from Theorem \ref{thm:por}. Finally, \eqref{areg} follows putting
\eqref{apor}--\eqref{apor'} together. 
\end{proof}

\begin{rems}
\emph{a)} Our choice to put one isolated point in the middle of each gap is
somewhat arbitrary. The Theorem \ref{thm:atomic} (and Lemma
\ref{lemma:exp}) holds true for many
other choices of (collections) of isolated points as well. For
instance, instead of choosing the middle point of each
$J\in\mathcal{J}$, one could
consider a Whitney decomposition $\mathcal{W}_J$ of $J$ and choose
all the midpoints of the elements of $\mathcal{W}_J$ to be the
collection of isolated points inside $J$. Doubling measures on this
kind of Whitney modification sets have been considered in \cite{kw},
and \cite{www}.\\
\emph{b)} In many situations, the technical assumption \eqref{alphansmall} (used only to prove Lemma \ref{apulemma}) may be omitted. For the middle interval midpoint sets $M(\alpha_n)$, for instance, the claims \eqref{athick'}--\eqref{apor'} in Theorem 
\ref{thm:atomic} hold also without this assumption. 
\end{rems}

Kaufman and Wu \cite{kw} have posed the following problem: Does there exist
a compact set $X\subset\R$ with $X=\overline{F_X}$ and a doubling
measure $\nu$ on $X$ so that $\nu|_{E_X}$ is a doubling measure on $E_X$? 
Recall that $F_X$ is the set of isolated points of $X$ and
$E_X=X\setminus F_X$.
The following example yields a positive answer to their question.
\begin{ex}\label{kwrem}
Let $(\alpha_n)\in\ell^1$, $X=M(\alpha_n)$, $C=C(\alpha_n)$,
$\mu=\Le|_{[0,1]}$, and let $\nu$ be 
a doubling measure on $X$ given by Lemma \ref{lemma:exp}. Then
$F_X=\cup_{J\in\mathcal{J}}\{x_J\}$, $E_X=C$, and 
$X=\overline{F_X}$. Moreover, it is easy to see that $\nu|_C=\Le|_C$
is a doubling measure on $C$ since 
there exists $c=c(\alpha_n)$ so that $\Le(C\cap(x-r,x+r))>cr$ for
all $x\in C$ and $0<r<1$. 
\end{ex}

\subsection{On sets with positive Lebesgue measure}

To complete the discussion on purely atomic doubling measures, we
answer a question posed by Lou, Wen, and Wu 
in \cite{lww}. As observed by Wu \cite[Example 1]{wu}, see also Lou, Wen, and Wu
\cite[Theorem 1]{lww}, it is possible to construct compact sets $X\subset[0,1]$
with Hausdorff dimension one so that all doubling measures on $X$ are
purely atomic. The examples of Wu \cite{wu} and Lou, Wen, and Wu
\cite{lww} are countable unions of self-similar Cantor sets whose
dimensions gets closer and closer to one. Another, more direct way to
obtain such a set is given by Theorem \ref{thm:atomic}: Choosing
$X=M(\alpha_n)$ for any sequence
$(\alpha_n)\notin\bigcup_{0<p<\infty}\ell^p$ such that 
\begin{equation}\label{eq:dim1}
\lim_{n\rightarrow\infty}\frac{\log\bigl(\prod_{k=1}^n(1-\alpha_k)\bigr)}{n}=0
\end{equation}
will do. Note that \eqref{eq:dim1} always holds if $\lim_{n\rightarrow\infty}\alpha_n=0$. It
was asked by Lou, Wen, and Wu \cite{lww} whether there are compact
sets $X\subset\R$ with positive Lebesgue measure so that all doubling
measures $\mu$ on $X$ are purely atomic. The answer is
negative.

\begin{prop}\label{vkprop}
If $X\subset \R$ is compact and $\Le(X)>0$, there are 
doubling measures on $X$ with nontrivial continuous part.
\end{prop}

\begin{proof}
The claim is a direct consequence of the results of Vol'berg
and Konyagin \cite{vk}, see also \cite[\S 13]{he}. For subsets of
$\Rn$, they proved the 
existence of $n$-homogeneous measures. In our case this gives a constant $c<\infty$ and a
measure $\mu$ on $X$ so that
\begin{equation*}
\mu(B(x,\lambda r))\leq c \lambda \mu(B(x,r))\text{ for all }x\in X,
\lambda \geq 1,\text{ and }r>0\,.
\end{equation*} 
Putting $\lambda=1/r$, it follows 
that $c\mu(B(x,r))\geq r$ for all $x\in X$ and
$0<r<1$. Now we may define $\nu=\mu+\Le|_X$. If $c'$ is the doubling constant of
$\mu$, it follows that for all $x\in X$ and $0<r<1$,
\begin{align*}
\nu(B(x,2r))&=\mu(B(x,2r))+\Le(X\cap B(x,2r))\leq \mu(B(x,2r))+2r\\
&\leq (c'+2c)\mu(B(x,r))\leq(c'+2c)\nu(B(x,r))
\end{align*}
so $\nu$ is a doubling measure on $X$. 
As $\Le(X)>0$, it follows that $\nu$ has a
nontrivial (absolutely) continuous part.
\end{proof}

\begin{rem}
While this paper was in preparation, there has been some independent
research on the topics of the last section.  Wang and Wen \cite{ww} have
constructed a set $X$ with the same properties as in Example
\ref{kwrem} and Lou and Wu \cite{lw} have also
observed that Proposition \ref{vkprop} follows from the above mentioned
result of Vol'berg and Konyagin.  
\end{rem}

\emph{Acknowledgements.} We are grateful to Kevin Wildrick
for useful discussions.

\bibliographystyle{alpha}
\bibliography{cantordoubling}

\end{document}